\documentclass[a4paper,12pt]{article}
\usepackage[english]{babel}
\usepackage{mathrsfs}  
\usepackage{algorithmicx}
\usepackage{algorithm}
\usepackage{algpseudocode}
\usepackage{geometry}    
\usepackage{tikz}
\usepackage[colorlinks,linkcolor=red,citecolor=green]{hyperref}
\usepackage{blindtext}
\usepackage[utf8]{inputenc}
\usepackage{bbm}
\usepackage{xcolor}
\usepackage{graphicx}
\usepackage{wrapfig}
\usepackage{caption}
\usepackage{subcaption}
\usepackage{multicol}
\usepackage{xspace,amsmath,amsthm,amssymb,enumerate,color,esint,paralist}
\usepackage{datetime}
\usepackage{nicefrac}
\usepackage{pgfplots}

\usepackage[capitalise,compress]{cleveref} 

\crefname{equation}{}{}

\usepackage{appendix}
\newtheorem{theorem}{Theorem}[section]
\newtheorem*{theorem*}{Theorem}
\newtheorem{lemma}[theorem]{Lemma}

\newtheorem{corollary}[theorem]{Corollary}
\newtheorem*{claim*}{Claim}

\newtheorem*{remark*}{Remark}

\newtheorem*{examples*}{Examples}

\newcommand{\R}{\mathbbm{R}}
\newcommand{\N}{\mathbbm{N}}
\newcommand{\C}{\mathbbm{C}}

\newcommand{\1}{\mathbbm{1}}
\renewcommand{\P}{\mathbbm{P}}\newcommand{\var}{\mathrm{Var}}
\newcommand{\E}{\mathbbm{E}}
\newcommand{\unif}{\mathfrak{u}}

\renewcommand{\gets}{\curvearrowleft}

\renewcommand{\epsilon}{\varepsilon}
\allowdisplaybreaks
\begin{document}
\title{
Multilevel Picard approximations for\\
McKean-Vlasov stochastic differential equations}

\author{Martin Hutzenthaler\thanks{Faculty of Mathematics, University of Duisburg-Essen, Essen, Germany; e-mail: \texttt{martin.hutzenthaler}\textcircled{\texttt{a}}\texttt{uni-due.de}}
\and
Thomas Kruse\thanks{Institute of Mathematics, University of Gie{\ss}en, Gie{\ss}en, Germany; e-mail: \texttt{thomas.kruse}\textcircled{\texttt{a}}\texttt{math.uni-giessen.de}}
\and 
Tuan Anh Nguyen\thanks{Faculty of Mathematics, University of Duisburg-Essen, Essen, Germany; e-mail: \texttt{tuan.nguyen}\textcircled{\texttt{a}}\texttt{uni-due.de}}
}

\maketitle
\maketitle
\makeatletter
\let\@makefnmark\relax
\let\@thefnmark\relax
\@footnotetext{\emph{Key words and phrases:}
curse of dimensionality, high-dimensional SDEs, 
high-dimensional semilinear McKean-Vlasov SDEs, multilevel Picard approximations, multilevel Monte Carlo method. }
\@footnotetext{\emph{AMS 2010 subject classification}: 65M75} 
\makeatother

\begin{abstract}
In the literature
there exist approximation methods for McKean-Vlasov stochastic differential equations which have a computational effort of order $3$.
In this article we introduce full-history recursive
multilevel Picard approximations for McKean-Vlasov stochastic
differential equations.
We prove that these MLP approximations have
 computational effort of order $2+$
which is essentially  optimal in high dimensions.
\end{abstract}


\allowdisplaybreaks\sloppy
\linespread{1.1}
\section{Introduction}
McKean \cite{McKean1966} introduced stochastic differential equations (SDEs)
whose coefficients depend on the distribution of the solution.
These McKean-Vlasov SDEs allow a stochastic representation of solutions
of nonlinear, possibly non-local partial
parabolic differential equations (PDEs) such as Vlasov's equation,
Boltzmann's equation, or Burgers' equation.
Moreover,
weakly dependent diffusions
converge to independent solutions of
McKean-Vlasov SDEs 
 as the system size tends to infinity.
This phenomenon was termed \emph{propagation of chaos} by Kac
\cite{Kac1956} and is well studied in the literature;
 see, e.g., \cite{McKean1967,Oelschlaeger1985,Gaertner1988,Sznitman1989,Hutzenthaler2012,ST19,hutzenthaler2020propagation}.

For simplicity we consider in this article the McKean-Vlasov SDE in \eqref{eq:McKV}
below
with additive noise
 whose drift coefficient depends linearly on the distribution
of the solution.
In the literature there exist a number of approximation methods
for the solution of \eqref{eq:McKV}.
A direct approach approximates the spatial integral in \eqref{eq:McKV}
with an average over weakly dependent versions of the solution
(resulting in weakly interacting diffusions)
and the temporal integral in \eqref{eq:McKV} with suitable Rieman sums
(Euler method).
The $L^2$-error of this approximation is of order $1/N$ 
as $\N=\{1,2,\ldots\}\ni N\to\infty$ if
we use $N^2$ interacting diffusions and $N$ time intervals
resulting in $N^5$ function evaluations of the drift coefficient.
Thus the computational effort for achieving $L^2$-error $\varepsilon\in(0,1)$
is of order $\varepsilon^{-5}$ as $\epsilon\to0$;
cf., e.g., \cite{antonelli2002rate,bossy2002rate,bossy1996convergence}.
This computational effort for achieving error $\varepsilon\in(0,1)$
can be reduced to order $\varepsilon^{-4}|\log(\varepsilon)|^3$ by replacing
averages by the multilevel Monte Carlo method
(cf.\ \cite[Theorem 4.5]{STT19} and, e.g., \cite{h01,g08b,LemairePages2017})
and to order $\varepsilon^{-3}$ by the antithetic multilevel Monte Carlo
method  (cf.\ \cite[Theorem 4.3]{ST19}).
In the very special case of ordinary differential equations with
an expectation in the driving function, where the plain vanilla
Monte Carlo method  has computational effort of order $\varepsilon^{-3}$,
\cite[Theorem 1.1]{Beck2021full} 
shows that this computational effort can be reduced to order $\varepsilon^{-2+}$.
In low dimensions, the spatial integral in \eqref{eq:McKV} can also be approximated 
e.g.\ by projections on function spaces
and then the computational effort can be reduced
to order $\epsilon^{-2}|\log(\epsilon)|^4$ (or better); cf., e.g., \cite[Theorem 4]{DST19}.
In high dimensions, the numerical approximation of
Lebesgue integrals 
with (deterministic) quadrature rules suffers
from the curse of dimensionality;
see \cite{sukharev1979optimal}.
The Monte Carlo method overcomes this curse and achieves a
$L^2$-error $\epsilon\in(0,1)$ with computational effort
of order $\epsilon^{-2}$ in the numerical approximation of Lebesgue integrals without the curse of dimensionality.
Thus, in high dimensions, the computational effort
for approximating the spatial integral
on the right-hand side of \eqref{eq:McKV} 
has optimal order $\epsilon^{-2}$ and this is clearly a
lower bound for the approximation of the full McKean-Vlasov SDE.
It remained an open question in the literature whether
McKean-Vlasov SDEs can be approximated
up to $L^2$-error $\epsilon\in(0,1)$
with computational effort of order $\epsilon^{-2}$
(or whether an higher effort such as $\epsilon^{-3}$
is required in general).

In this article we partially answer this question positively.
In other words, we show that the computational problem of approximating
the solution
of the McKean-Vlasov SDE in \eqref{eq:McKV} has up to logarithmic
factors the
same computational complexity as the numerical approximation
of the spatial integral in \eqref{eq:McKV}.
More specifically,
we view \eqref{eq:McKV} as fixed point equation
and
adapt the full-history multilevel Picard (MLP) method,
which was introduced in \cite{EHutzenthalerJentzenKruse2016},
to this fixed point equation.
This MLP method was already successfully applied to overcome
the curse of dimensionality in the numerical approximation
of semilinear PDEs; see, e.g., \cite{HJK+18,hjk2019overcoming,beck2019overcoming,HJKN20}.
Our
MLP approximation method \eqref{eq:MLP} below is, roughly speaking,
 based on the idea to (a) reformulate
the McKean-Vlasov SDE in \eqref{eq:McKV}
 as a stochastic fixed point problem
$X=\Phi(X)$ with a suitable function
$\Phi$,
to (b) approximate the fixed point $X$ through Picard iterates $(X_k)_{k\in\{0,1,2,\ldots\}}$,
to (c) write $X$ as telescoping series over this sequence, that is,
\begin{equation}  \begin{split} \label{eq:telescope}
  X=X_1+\sum_{k=1}^{\infty}(X_{k+1}-X_k)
   =X_1+\sum_{k=1}^{\infty}\Big(\Phi(X_{k})
-\Phi(X_{k-1})\Big),
\end{split}     \end{equation}
and to (d) approximate the series by a finite sum and the
temporal and spatial integrals
in the summands by Monte Carlo averages with fewer and fewer
 independent samples
as $k$ increases.
Roughly speaking, the rationale behind this approach is that
 $X_{k+1}-X_k$
converges exponentially fast (or even factorially fast)
 to $0$ as $k\to\infty$
and the mean squared error of the Monte Carlo average is bounded by the
second moment of the involved random variable divided by the number of
independent samples in the average.
This motivates our MLP approximations in~\eqref{eq:MLP}.

The main result of this article, Theorem \ref{t01}
in Section \ref{sec:MLP} below,
 implies that the MLP approximation method
approximates solutions of McKean-Vlasov SDEs with additive noise
 whose drift coefficients depend linearly on the distribution
of the solution
up to an $L^2$-error $\epsilon\in(0,1)$
with computational effort $\epsilon^{-2+}$
without suffering from the curse of dimensionality.
To illustrate our main results,
 we now present in Theorem \ref{t00}
 a special case of Theorem \ref{t01}.
\begin{theorem}\label{t00}
Let $\delta, T\in (0,\infty)$,
$d\in\N$,
$\xi\in\R^d$,
$\Theta= \bigcup_{n\in\N}(\N_0)^n$,
let $\lVert \cdot\rVert\colon \R^d\to[0,\infty)$ 
be a norm,
let
 $\mu \colon \R^d\times\R^d\to\R^d$ be globally Lipschitz continuous,
let $(\Omega,\mathcal{F},\P)$ be a  probability space, 
let $\unif^\theta\colon\Omega\to[0,1]$, $\theta\in\Theta$, be i.i.d.\ random variables, assume for all $t\in [0,1] $ that
$\P(\unif^0\leq t)=t$,
let $W^\theta\colon [0,T]\times\Omega\to\R^d$, $\theta\in\Theta$, be i.i.d.\ standard 
Brownian motions with continuous sample paths, 
assume that 
$(\unif^\theta)_{ \theta\in\Theta}$ and
$(W^\theta)_{ \theta\in\Theta}$ 
are independent,
let $X^\theta_{n,m}\colon [0,T]\times\Omega\to\R^d$, $\theta\in\Theta$, $n,m\in\N_0$,  satisfy for all
$\theta\in\Theta$, $m\in\N$, $n\in \N_0$, $t\in [0,T] $  that
\begin{align}
\begin{split}\label{eq:MLP}
&X_{n,m}^\theta(t)= \left(\xi+ W^\theta\left(\max\!\left(\{\tfrac{kT}{m^n}\colon k\in\N_0\}\cap[0,t]\right)\right)+t\mu(0,0)\right)\1_{\N}(n)\\&+ \sum_{\ell=1}^{n-1}
\sum_{k=1}^{m^{n-\ell}}\tfrac{t\left[
\mu\bigl(X^\theta_{\ell,m}(\unif^{(\theta,n,k,\ell)}t),
X^{(\theta,n,k,\ell)}_{\ell,m}(\unif^{(\theta,n,k,\ell)}t)
\bigr)
-
\mu\bigl(X^\theta_{\ell-1,m}(\unif^{(\theta,n,k,\ell)}t),
X^{(\theta,n,k,\ell)}_{\ell-1,m}(\unif^{(\theta,n,k,\ell)}t)
\bigr)\right]}{m^{n-\ell}},
\end{split}\end{align}
let
$X\colon [0,T]\times\Omega\to\R^d$ be a $(\sigma(\{W^0(s)\colon s\in [0,t]\}))_{t\in[0,T]}$-adapted
stochastic process 
with continuous sample paths, assume for all $t\in[0,T]$  that
$\int_0^T\left(\E\!\left[\lVert X(s)\rVert^2\right]\right)^{\!\nicefrac{1}{2}}\,ds<\infty$
 and
\begin{align}\label{eq:McKV}
X(t)=\xi+\int_0^t \int \mu(X(s),x)\P\bigl(X(s)\in dx\bigr)\,ds+W^0(t),
\end{align}
and
for every $n,m\in\N $
let
$C_{n,m}\in \N_0$
be the number of function evaluations of $\mu$ and the number of scalar random variables which are used to compute one realization of $X_{n,m}^0(T)$
(cf.~\eqref{a06} in~\cref{t01} below).
 Then there exist
 $c\in\R$
 and
  $\mathrm{n}=(\mathrm{n}_\epsilon)_{\epsilon\in (0,1]}\colon (0,1] \to \N$,
such that for all
$\epsilon\in(0,1]$ it holds that
$ \sup_{k\in[\mathrm{n}_\epsilon,\infty)\cap\N}\sup_{t\in[0,T]}\left(\E \!\left[
\lVert 
X^0_{k,k}(t)-X(t)\rVert^2\right]\right)^{1/2}\leq \epsilon$ and $C_{\mathrm{n}_\epsilon,\mathrm{n}_\epsilon}\le c \, \epsilon^{-(2+\delta)}$.
\end{theorem}

In the following we add further comments on our approximation method.
The MLP approximations
  ${X}_{ n,m}^{\theta} $, 
$n\in\N_0$, $m\in\N$, $\theta\in\Theta$, in \eqref{eq:MLP}
are indexed by the number $n\in\N_0$ of fixed point iterates,
by a parameter $m\in\N$ which is fixed in the recursion
 in \eqref{eq:MLP} and is the basis of the number of Monte Carlo
averages,
and by a parameter $\theta\in\Theta$ which is
used to distinguish independent
MLP approximations in \eqref{eq:MLP}.
We note for every $\theta$ that all
$X_{n,m}^{\theta}$, $n,m\in\N$, depend on the same Brownian path
$W^{\theta}$ so that for all $n,m\in\N$, $\theta\in\Theta$, $t\in[0,T]$
we need to have
$(W^{\theta}(\tfrac{kT}{m^n}))_{k\in\{0,1,\ldots,m^n\}}$ 
as argument of the function call which calculates
$X_{n,m}^{\theta}(t)$.

The remainder of this article is organized as follows.
In Section \ref{sec:recursions} we solve recursions of
Gronwall-type. In particular, Corollary \ref{l05} will be applied
to obtain an upper bound for the computational effort
which satisfies the recursion in \eqref{a06} in Theorem~\ref{t01}.
Moreover, in Theorem \ref{t01} in Section \ref{sec:MLP}
we estimate the $L^2$-error between the solution of the McKean-Vlasov SDE
and our MLP approximations
and we estimate the computational effort for computing
one realization of our MLP approximation.

\section{Discrete Gronwall-type recursions}\label{sec:recursions}
In this section we solve recursions of Gronwall-type.
The following result, Lemma~\ref{r01}, provides the exact solutions
of certain linear recurrence relations of second order.

\begin{lemma}[Two-step recursions]\label{r01}
Let $\kappa ,\lambda,x_1,x_2 \in \C$, $
(a_k)_{k\in\N_0},(b_k)_{k\in \N_0}\subseteq \C$  
satisfy for all $k\in \N_0$, $i\in \{1,2\}$ that
\begin{align}\label{r02}
a_0=b_0, \quad
a_1=b_1+\kappa b_0,
\quad a_{k+2}=b_{k+2}+\kappa a_{k+1}+\lambda a_{k},\quad 
x_i^2=\kappa  x_i+\lambda,\quad\text{and}\quad
x_1\neq x_2.
\end{align}
Then it holds for all $k\in \N_0$ that $a_{k}=\frac{1}{x_2-x_1}\sum_{\ell=0}^{k}b_\ell(x_2^{k-\ell+1}-x_1^{k-\ell+1}).$
\end{lemma}
\begin{proof}[Proof of \cref{r01}]Throughout this proof let $(z_k)_{k\in\N_0}\subseteq \C$ satisfy for all $k\in \N_0$ that
$
z_{k}=\frac{1}{x_2-x_1}\sum_{\ell=0}^{k}
b_\ell \left( x_2^{k-\ell+1}- x_1^{k-\ell+1}\right)
$. We consider the two cases $\lambda=0$ and $\lambda\neq 0$.

\emph{Case 1.} $\lambda=0$.
The fact that $x_1\neq x_2$ and
the fact that
$\forall\,i\in\{1,2\}\colon x_i^2=\kappa  x_i$ prove that  
$\kappa\neq 0 $ and 
$(x_1,x_2)\in \{(0,\kappa),
(\kappa,0)\}$.
This proves  for all $k\in \N_0$  that
$z_k=\tfrac{1}{\kappa }\sum_{\ell=0}^{k} b_\ell \kappa ^{k-\ell+1}= \sum_{\ell=0}^{k} b_\ell \kappa ^{k-\ell}$. This and the fact that $\lambda=0$ show for all
$k\in\N_0$ that
$z_0=b_0$,
$z_1= b_1+\kappa b_0$, and
$z_{k+2}= \sum_{\ell=0}^{k+2} b_\ell \kappa ^{k+2-\ell}
=b_{k+2}+\kappa z_{k+1}= b_{k+2}+\kappa z_{k+1}
+\lambda z_{k}.$ 
This,  \eqref{r02}, and induction prove  for all $k\in\N_0$ that 
$z_k=a_k$. 

\emph{Case 2.} $\lambda \neq 0$. The fact that $\forall\,i\in \{1,2\}\colon x_i^2=\kappa  x_i+\lambda$ implies that $x_1\neq 0$ and $x_2\neq 0$.
Moreover, 
the fact that
$0=(x_1^2 - \kappa x_1-\lambda)-(x_2^2 - \kappa x_2-\lambda)= (x_1-x_2)(x_1+x_2-\kappa)  $ and the fact that
$x_1\neq x_2$ imply that $x_1+x_2=\kappa$.
Next, the fact that
$0=\frac{x_1^2 - \kappa x_1-\lambda}{x_1} -\frac{x_2^2 - \kappa x_2-\lambda}{x_2}=x_1-x_2+\lambda(\frac{1}{x_2}-\frac{1}{x_1})$ imply that
$\frac{\lambda}{x_2-x_1}(\frac{1}{x_2}-\frac{1}{x_1})=1$.
This, the definition of $ (z_k)_{k\in\N_0}$, and the fact that $x_1+x_2=\kappa$ 
 imply for all $k\in \N_0$ that
$z_0=b_0=a_0$, 
$z_1=\frac{1}{x_2-x_1}\left(b_0(x_2^2-x_1^2)+b_1(x_2-x_1)\right)=b_0(x_1+x_2)+b_1=b_0 \kappa+b_1 =a_1$, and
\begin{align}
&z_{k+2}=
\sum_{\ell=0}^{k+2}\tfrac{b_\ell\left( x_2^{k-\ell+3}- x_1^{k-\ell+3}\right)}{x_2-x_1}
=
\sum_{\ell=0}^{k+2}\tfrac{b_\ell\left( x_2^{k-\ell+1}x_2^2- x_1^{k-\ell+1}x_1^2\right)}{x_2-x_1}
=\sum_{\ell=0}^{k+2}\tfrac{b_\ell \left(x_2^{k-\ell+1}(\kappa x_2+\lambda )-x_1^{k-\ell +1}(\kappa x_1+\lambda )\right)}{x_2-x_1}\nonumber \\
&=\kappa \left[\sum_{\ell=0}^{k+2}\tfrac{b_\ell \left(x_2^{k-\ell+2}-x_1^{k-\ell +2}\right)}{x_2-x_1}
\right]
+\lambda \left[
\sum_{\ell=0}^{k+2}\tfrac{b_\ell \left(x_2^{k-\ell+1}-x_1^{k-\ell +1}\right)}{x_2-x_1}\right]\nonumber \\
&=\kappa\left[\sum_{\ell=0}^{k+1}
\tfrac{b_\ell \left( x_2^{(k+1)-\ell+1}- x_1^{(k+1)-\ell+1}\right)}{x_2-x_1}\right]
+
\lambda \left[
\sum_{\ell=0}^{k}\tfrac{b_\ell \left(x_2^{k-\ell+1}-x_1^{k-\ell +1}\right)}{x_2-x_1}\right]
+\lambda \left[\sum_{\ell=k+1}^{k+2}\tfrac{b_\ell \left(x_2^{k-\ell+1}-x_1^{k-\ell +1}\right)}{x_2-x_1}\right]
\nonumber \\
&=\kappa  z_{k+1}+\lambda  z_{k}+\tfrac{\lambda  b_{k+2}}{x_2-x_1}\left(\tfrac{1}{x_2}-\tfrac{1}{x_1}\right)=\kappa z_{k+1}+\lambda z_{k}+b_{k+2}.
\end{align}
This, \eqref{r02}, and induction 
show  for all $k\in\N_0$ that 
$z_k=a_k$.  Combining the two cases $\lambda=0$ and $\lambda\neq 0$
 completes the proof of \cref{r01}.
\end{proof}

The following result, Lemma~\ref{l01}, generalizes
the discrete Gronwall inequality which is the special
case $\lambda=0$ of  Lemma~\ref{l01}.
\begin{lemma}[Discrete Gronwall-type recursion] \label{l01}
Let $\kappa ,\lambda , x_1,x_2\in\C$, $(a_n)_{n\in\N_0},(b_n)_{n\in\N_0}\subseteq\C$ 
satisfy for all $n\in\N_0$, $i\in \{1,2\}$ that 
\begin{align}
\label{r08}
\begin{split}
 a_{n}=b_n+\sum_{k=0}^{n-1}\bigl[\kappa  a_k+\1_{\N}(k)\lambda  a_{\lvert k-1\rvert}\bigr],\quad
x_i^2=(1+\kappa )x_i+\lambda   ,
\quad\text{and}\quad
x_1\neq x_2    .
\end{split}
\end{align}
  Then it holds for all $n\in\N_0$ that
  \begin{equation}  \begin{split} \label{eq:exp.recursion.assertion}
    a_n=\sum_{k=0}^{n}\left[
\tfrac{b_k-\1_{\N}(k)b_{\lvert  k-1\rvert}}{x_2-x_1}\left(
x_2^{n-k+1}-x_1^{n-k+1}\right)\right].
  \end{split}     \end{equation}
\end{lemma}
\begin{proof}[Proof of Lemma~\ref{l01}]Throughout this proof let $(z_n)_{n\in\N_0}\subseteq \C$ satisfy for all $n\in\N_0$ that $z_n=\sum_{k=0}^{n}a_k$.
This and \eqref{r08} show for all $n\in\N_0$ that
$z_0=a_0=b_0$,
$z_1= a_0+a_1= a_0+b_1+\kappa  a_0=b_1+(1+\kappa )b_0$,
$
  z_{n+2}-z_{n+1}=a_{n+2}
=
b_{n+2}+
\kappa \bigl(\sum_{k=0}^{n+1} a_k\bigr)+\lambda \bigl(\sum_{k=0}^{n}a_k\bigr)
     =b_{n+2}+\kappa  z_{n+1}+\lambda  z_{n}$, and therefore 
$z_{n+2}=b_{n+2}+(1+\kappa ) z_{n+1}+\lambda  z_{n}.$
This, 
\cref{r01} (applied with
$\kappa\gets (1+\kappa )$,
 $(a_n)_{n\in\N_0}\gets (z_n)_{n\in\N_0}$
in the notation of \cref{r01}), and the assumptions on $x_1,x_2$
prove for all $n\in\N_0$ that
$
z_{n}=\sum_{k=0}^{n}
\frac{b_k(x_2^{n-k+1}-x_1^{n-k+1})}{x_2-x_1}
$. Therefore, it holds for all $n\in\N$ that
\begin{align}\begin{split}
&
a_n= z_n-z_{n-1}=
\sum_{k=0}^{n}
\tfrac{b_k(x_2^{n-k+1}-x_1^{n-k+1})}{x_2-x_1}
-\sum_{k=0}^{n-1}
\tfrac{b_k(x_2^{n-k}-x_1^{n-k})}{x_2-x_1}\\
&=
\sum_{k=0}^{n}
\tfrac{b_k(x_2^{n-k+1}-x_1^{n-k+1})}{x_2-x_1}
-\sum_{k=1}^{n}
\tfrac{b_{k-1}(x_2^{n-k+1}-x_1^{n-k+1})}{x_2-x_1}=
\sum_{k=0}^{n}
\tfrac{(b_k-\1_{\N}(k)b_{\lvert k-1\rvert})(x_2^{n-k+1}-x_1^{n-k+1})}{x_2-x_1}.
 \end{split}
\end{align}
This and the fact that
$a_0=b_0$ (see
\eqref{r08}) complete the proof of \cref{l01}.
\end{proof}

The following result, Corollary~\ref{l05}, generalizes, e.g., 
\cite[Lemma 3.6]{HJK+18}.
\begin{corollary}[Discrete Gronwall-type inequality]\label{l05}
Let $(a_n)_{n\in\N_0}\subseteq [0,\infty)$, $\kappa,\lambda, c_1,c_2, c_3,c_4,\beta $ $\in [0,\infty)$ satisfy for all $n\in\N_0$ that
\begin{align}\label{l03}
a_n\leq  c_1 +c_2 n + c_3 \sum_{k=1}^{n} c_4^k
+
 \sum_{k=0}^{n-1} \left[
\kappa a_k +\lambda \1_{\N}(k)a_{\lvert k-1\rvert}\right]\quad\text{and}\quad\beta=\tfrac{(1+\kappa)+\sqrt{(1+\kappa)^2+4\lambda}}{2}>1.
\end{align}
Then
it holds for all $n\in\N_0$ that
\begin{align}
a_n\leq \begin{cases}
\tfrac{3}{2} \beta ^n c_1
+ \tfrac{3c_2(\beta^n-1)}{2( \beta-1)}
+\tfrac{3}{2} c_3n
\beta^{n}
&\text{if } c_4=\beta,\\
\tfrac{3}{2} \beta ^n c_1
+\tfrac{3c_2(\beta^n-1)}{2( \beta-1)}
+\tfrac{3c_3( c_4^{n+1}- c_4\beta ^n)}{2( c_4- \beta) }&\text{else. }
\end{cases}
\end{align}

\end{corollary}
\begin{proof}[Proof of \cref{l05}]Throughout this proof let 
$(\tilde{a}_n)_{n\in\N_0},
(b_n)_{n\in\N_0}
\subseteq [0,\infty)$, $x_1,x_2\in \R$ satisfy for all
$n\in\N_0$ that
$
x_1= \tfrac{(1+\kappa)-\sqrt{(1+\kappa)^2+4\lambda}}{2}$,
$ 
x_2= \tfrac{(1+\kappa)+\sqrt{(1+\kappa)^2+4\lambda}}{2}$,
\begin{align}
\tilde{a}_n=  b_n +
 \sum_{k=0}^{n-1} \left[
\kappa \tilde{a}_k +\lambda \1_{\N}(k)\tilde{a}_{\lvert k-1\rvert}\right],\quad\text{and}\quad
b_n=
c_1 +c_2n + c_3 \sum_{k=1}^{n} c_4^k.\label{l04}
\end{align}
This and the quadratic formula show for all $n\in \N_0$, 
$k\in\{0,1,\ldots,n\}$,
$i\in \{1,2\}$ that
\begin{align}\begin{split}
&
\tfrac{\lvert x_1\rvert}{\lvert  x_2-x_1\rvert}= \tfrac{
\sqrt{(1+\kappa)^2+4\lambda}- (1+\kappa)}{2\sqrt{(1+\kappa)^2+4\lambda}}\leq \tfrac{1}{2}
,\quad 
\tfrac{\lvert x_2\rvert}{\lvert  x_2-x_1\rvert}
= \tfrac{(1+\kappa)+\sqrt{(1+\kappa)^2+4\lambda}}{2\sqrt{(1+\kappa)^2+4\lambda}}\leq 1
,
\\
\lvert 
x_1\rvert\leq \lvert x_2\rvert,
&\quad\text{and}\quad 
\left\lvert \tfrac{x_2^{n-k+1}-x_1^{n-k+1}}{x_2-x_1}\right\rvert
\leq 
\left[ 
\tfrac{\lvert x_2\rvert}{\lvert  x_2-x_1\rvert}\lvert x_2\rvert ^{n-k}+
\tfrac{1}{\lvert  x_2-x_1\rvert}\lvert x_1\rvert ^{n-k+1} \right]
\leq \tfrac{3}{2} \lvert x_2\rvert^{n-k},
\end{split}
\end{align}
$x_1\neq x_2$,
$x_i^2= (1+\kappa)x_i+\lambda$,
$b_0=c_1$,
and 
$b_{n+1}-b_{n}=c_2+c_3c_4^{n+1} $. This, \cref{l01} (applied with $ (a_n)_{n\in\N_0}\gets (\tilde{a}_n)_{n\in\N_0}$ in the notation of \cref{l01}),  the fact that $\lambda,\kappa,c_1,c_2,c_3,c_4\geq 0$, and the definition of $x_2$ show for all $n\in\N_0$ that
\begin{align}\begin{split}
&\tilde{a}_n= \sum_{k=0}^{n}
\tfrac{(b_k-\1_{\N}(k)b_{\lvert k-1\rvert})(x_2^{n-k+1}-x_1^{n-k+1})}{x_2-x_1}
\leq \tfrac{3}{2}c_1 \lvert x_2\rvert^n+
\sum_{k=1}^{n}\Bigl[\tfrac{3}{2}(c_2+c_3 c_4^{k})\lvert x_2 \rvert^{n-k}\Bigr]\\
&=\tfrac{3}{2}\lvert x_2 \rvert^n c_1
+\tfrac{3}{2}c_2\left[\sum_{k=0}^{n-1} \lvert x_2\rvert^{k}\right]
+ 
\tfrac{3}{2}c_3\lvert x_2 \rvert^{n}\left[
\sum_{k=1}^{n}\tfrac{\lvert c_4\rvert^{k}}{ \lvert x_2\rvert^{k}}\right].
\end{split}
\end{align}
This and the fact that
$x_2=\beta >1$
show for all $n\in\N_0$ that if $c_4\neq x_2$, then
\begin{align}\begin{split}
&\tilde{a}_n \leq \tfrac{3}{2}\lvert x_2 \rvert^n c_1
+\tfrac{3}{2} c_2\tfrac{\lvert x_2\rvert^n-1}{\lvert x_2\rvert-1}
+\tfrac{3}{2} c_3\lvert x_2 \rvert^{n}
\frac{\lvert \frac{c_4}{x_2}\rvert^{n+1}-\lvert \frac{c_4}{x_2}\rvert}{\lvert \frac{c_4}{x_2}\rvert-1}
\\
&
\leq \tfrac{3}{2}\lvert x_2 \rvert^n c_1
+\tfrac{3}{2} c_2\tfrac{\lvert x_2\rvert^n-1}{\lvert x_2\rvert-1}
+\tfrac{3}{2} c_3
\tfrac{\lvert c_4\rvert^{n+1}-c_4\lvert x_2\rvert ^n}{\lvert c_4\rvert-\lvert x_2\rvert }
\label{l06}
\end{split}
\end{align}
and if $c_4=x_2$, then \begin{equation}
\tilde{a}_n\leq \tfrac{3}{2}\lvert x_2 \rvert^n c_1
+\tfrac{3}{2} c_2\tfrac{\lvert x_2\rvert^n-1}{\lvert x_2\rvert-1}
+\tfrac{3}{2} c_3n
\lvert x_2 \rvert^{n}
.\label{l06b}
\end{equation}
Furthermore, \eqref{l03}, \eqref{l04}, and induction prove for all $n\in \N_0$ that $a_n\leq \tilde{a}_n$. This,  \eqref{l06}, \eqref{l06b}, and the fact that $x_2=\beta$ complete the proof of \cref{l05}.
\end{proof}

\section{Multilevel Picard approximations  of McKean-Vlasov SDEs}
\label{sec:MLP}
The following theorem, Theorem~\ref{t01}, shows that the computational effort of MLP  approximations
of McKean-Vlasov SDEs is of order $2+$ if the noise is additive
and if the drift coefficients depend linearly on the distributions.
In Theorem~\ref{t01}, for every $n,M\in \N$ we think of $C_{n,M}$ as an upper bound for the sum of the number
of scalar 
random variables 
and the number of function evaluations of the drift coefficient 
which are used to compute one realization of
$X_{n,m}^0 (T)$.
Let us comment on the recursion \eqref{a06} which describes this computational effort. The binary variables $\mathfrak{v}, \mathfrak{f}\in \{0,1\}$ indicate whether we want to count the number of scalar random variables ($\mathfrak{v}=1$) and whether we want to count the number of function evaluations of the drift coefficient ($\mathfrak{f}=1$). 
For every $n,M\in \N$ to compute one realization of
$X_{n,m}^0 (T)$ the scheme in \eqref{a05} first has to generate a realization of $(W^0(\frac{kT}{m^n}))_{k\in \{0,1,\ldots,m^n\}}$ which corresponds to the generation of $m^nd$ scalar random variables. Additionally, the scheme evaluates the drift coefficient once at $(0,0)\in \R^d\times \R^d$. Next, for every $l\in \{1,2,\ldots, n-1\}$ the scheme does $m^{n-l}$ times the following: it evaluates the drift coefficient twice, it generates a continuously uniformly on $[0,1]$ distributed random variable, it generates a realization of 
$(W^\theta(\frac{kT}{m^l}))_{k\in \{0,1,\ldots,m^l\}}$ (corresponding to  $m^ld$ scalar random variables), and for suitable $s\in [0,T]$, $\theta\in \Theta$ it calls twice the functions which calculate 
$X_{l,m}^\theta (s)$ and $X_{n,m}^\theta (s)$.

\begin{theorem}\label{t01}
Let 
$T,L\in (0,\infty)$, 
$d\in\N$,
$\xi\in\R^d$, $\mu \in C(\R^d\times\R^d,\R^d)$, 
$\Theta= \bigcup_{n\in\N}(\N_0)^n$,
let $\lVert \cdot\rVert\colon \R^d\to[0,\infty)$ 
be the standard norm,
assume for all $x_1,y_1,x_2,y_2\in\R^d$ that
\begin{align}\label{a03}
\bigl\lVert \mu(x_1,y_1)-\mu(x_2,y_2)\bigr\rVert\leq \tfrac{L}{2}\lVert x_1-x_2\rVert+ \tfrac{L}{2}\lVert y_1-y_2\rVert,
\end{align}
let $(\Omega,\mathcal{F},\P)$ be a  probability space, 
let $\unif^\theta\colon\Omega\to[0,1]$, $\theta\in\Theta$, be i.i.d.\ random variables, assume for all $t\in [0,1] $ that
$\P(\unif^0\leq t)=t$,
let $W^\theta\colon [0,T]\times\Omega\to\R^d$, $\theta\in\Theta$, be i.i.d.\ standard 
Brownian motions with continuous sample paths, 
assume that 
$(\unif^\theta)_{ \theta\in\Theta}$ and
$(W^\theta)_{ \theta\in\Theta}$ 
are independent,
let
$X\colon [0,T]\times\Omega\to\R^d$ be a $(\sigma(\{W^0(s)\colon s\in [0,t]\}))_{t\in[0,T]}$-adapted
stochastic process 
with continuous sample paths, assume for all $t\in[0,T]$  that
$\int_0^T\left(\E\!\left[\lVert X(s)\rVert^2\right]\right)^{\!\nicefrac{1}{2}}\,ds<\infty$
 and
\begin{align}
\label{a02}
X(t)=\xi+\int_0^t \int \mu(X(s),x)\P\bigl(X(s)\in dx\bigr)\,ds+W^0(t),
\end{align}
let $X^\theta_{n,m}\colon [0,T]\times\Omega\to\R^d$, $\theta\in\Theta$, $m\in\N$, $n\in \N_0$, satisfy  for all
$\theta\in\Theta$, $m,n\in\N$, $t\in [0,T] $  that
$
X_{0,m}^\theta(t)=0
$
and \begin{align}\label{a05}
\begin{split}
&X_{n,m}^\theta(t)= \xi+ W^\theta\left(\sup\!\left(\{\tfrac{kT}{m^n}\colon k\in\N_0\}\cap[0,t]\right)\right)+t\mu(0,0)\\&+ \sum_{\ell=1}^{n-1}
\sum_{k=1}^{m^{n-\ell}}\tfrac{t\left[
\mu\bigl(X^\theta_{\ell,m}(\unif^{(\theta,n,k,\ell)}t),
X^{(\theta,n,k,\ell)}_{\ell,m}(\unif^{(\theta,n,k,\ell)}t)
\bigr)
-
\mu\bigl(X^\theta_{\ell-1,m}(\unif^{(\theta,n,k,\ell)}t),
X^{(\theta,n,k,\ell)}_{\ell-1,m}(\unif^{(\theta,n,k,\ell)}t)
\bigr)\right]}{m^{n-\ell}},
\end{split}\end{align}
let 
$\mathfrak{v},\mathfrak{f}\in\{0,1\}$,
and let
$C_{n,m}\in \N_0$, $m,n\in\N_0$, satisfy for all
$m,n\in\N$
 that
$C_{0,m}=0$ and
\begin{equation}
C_{n,m} \leq 
\mathfrak{v}m^nd+\mathfrak{f}+
\sum_{\ell=1}^{n-1}\Bigl[ m^ {n-\ell}\Bigl( \mathfrak{v}(m^\ell d+1) +2 \mathfrak{f}
+2 C_{\ell,m}+2  C_{\ell-1,m}\Bigr)\Bigr].\label{a06}
\end{equation}
Then\begin{enumerate}[(i)]\itemsep0pt
\item \label{t01c}it holds for all 
$t\in[0,T]$ that
$
\left(\E\!\left[\lVert X(t)\rVert^2\right]\right)^{\!\nicefrac{1}{2}}
\leq  \left[
\lVert \xi\rVert+ \lVert \mu(0,0)\rVert t+
\sqrt{td}\right]e^{Lt}$,
\item \label{t01a}
it holds
for all
$t\in[0,T]$,
$m,n\in\N$  that
$X^0_{n,m}(t)$ is measurable and
\begin{equation}
\left(
\E \!\left[\left
\lVert 
X^0_{n,m}(t)-X(t)\right\rVert^2\right]\right)^{\!\nicefrac{1}{2}}
\leq m^{-n/2}e^{m/2}
\left[
\lVert \xi\rVert+ \lVert \mu(0,0)\rVert t+
\sqrt{Td}\right]e^{Lt}
\left(1+2Lt\right)^{n},
\end{equation}
and 
\item\label{t01b} there exists $\mathrm{n}=(\mathrm{n}_\epsilon)_{\epsilon\in (0,1)}\colon (0,1) \to \N$ such that for all
$\delta,\epsilon\in(0,1)$  it holds that
$ \sup_{k\in[\mathrm{n}_\epsilon,\infty)\cap\N}\sup_{t\in[0,T]}\E \!\left[
\lVert 
X^0_{k,k}(t)-X(t)\rVert^2\right]\leq \epsilon^2$ and 
\begin{equation}\small\begin{split}
C_{\mathrm{n}_\epsilon,\mathrm{n}_\epsilon}\epsilon^{2+\delta}\leq  (\mathfrak{v}d+ \mathfrak{f})
\sup_{k\in\N}\left[
(4k+4)^{k+1}\left[\tfrac{
e^{k/2}
\left[1+
\lVert \xi\rVert+ \lVert \mu(0,0)\rVert T+
\sqrt{Td}\right]e^{LT}
\left(1+2LT\right)^{k}}{k^{k/2}}\right]^{2+\delta}\right]<\infty .
\end{split}
\end{equation}
\end{enumerate}
\end{theorem}
\begin{proof}[Proof of \cref{t01}]\newcounter{step}
Throughout this proof for every random variable $\mathfrak{X}\colon \Omega\to\R^d$ with 
$\E [\lVert \mathfrak{X}\rVert^2]<\infty$
and every $\sigma$-algebra $\mathcal{A}\subseteq  \mathcal{F}$ let 
$\var[\mathfrak{X}|\mathcal{A}]\colon \Omega\to [0,\infty)$ be a random variable which satisfies that 
a.s.\ it holds that
$\var[\mathfrak{X}|\mathcal{A}]= 
\E \bigl[\lVert \mathfrak{X}
-\E [ \mathfrak{X}|\mathcal{A}]
\rVert^2|\mathcal{A}\bigr]
$ and let $\mathcal{G}_{n,m}\subseteq \mathcal{F}$, $n\in\N_0$, 
$m\in\N$, be the $\sigma$-algebras which
satisfy for all 
$n\in\N_0$, $m\in\N$ that
$
\mathcal{G}_{n,m}=\sigma\!\left(\left\{W^0(t),X_{\ell,m}^0(t)\colon \ell\in \{0,1,\ldots,n\},t\in [0,T]\right\}\right).
$

This proof is organized as follows. 
In Step \ref{st05} we prove  the upper bound of the exact solution $X$ in~\eqref{t01c}.
In Steps~\ref{st01} and \ref{st02} we establish distributional, measurablility, and integrability properties for the MLP approximations in \eqref{a05}.
In Step~\ref{st03} we consider 
 \emph{the bias}. In
Step~\ref{st04} we consider \emph{the statistical error}.
In Step~\ref{st06} we  combine Steps~\ref{st03} and  \ref{st04} to obtain a recursive bound of the approximation error,
which, together with a Gronwall-type inequality and the upper bound of the exact solution in~\eqref{t01c}, establishes \eqref{t01a}. 
In Step~\ref{st07} we estimate the computational complexity and obtain \eqref{t01b}.

\refstepcounter{step}\label{st05} 
 \emph{Step~\thestep.}\ 
We prove the upper bound of the exact solution $X$ in  \eqref{t01c}.
Observe that 
Jensen's inequality,
the triangle inequality, and \eqref{a03} show  for all 
$s\in[0,T]$ that
\begin{equation}\begin{split}
&
\left(
\E\!\left[\left\lVert \int \mu(X(s),x)\P\bigl(X(s)\in dx\bigr)\right\rVert^2\right]\right)^{\!\nicefrac{1}{2}}
\leq 
\left(
\E\!\left[ \int \left\lVert\mu(X(s),x)\right\rVert^2\P\bigl(X(s)\in dx\bigr)\right]\right)^{\!\nicefrac{1}{2}}\\
&\leq \lVert \mu(0,0)\rVert
+\tfrac{L}{2} 
\left(\E\!\left[ \lVert X(s)\rVert^2\right]\right)^{\!\nicefrac{1}{2}}
+\tfrac{L}{2} 
\left(\E\!\left[ \lVert X(s)\rVert^2\right]\right)^{\!\nicefrac{1}{2}}\\
&
= \lVert \mu(0,0)\rVert+
L\left(\E\! \left[\lVert X(s)\rVert^2\right]\right)^{\!\nicefrac{1}{2}}.
\end{split}\end{equation}
This, \eqref{a02},  the triangle inequality, and the fact that
$\forall\,t\in[0,T] \colon \E\bigl[\left\lVert W^0(t)\right\rVert^2\bigr]=td$ 
 prove  for all $t\in[0,T]$  that
\begin{align}\begin{split}
& \left(\E\!\left[\lVert X(t)\rVert^2\right]\right)^{\!\nicefrac{1}{2}}
=
 \left(\E\!\left[\left \lVert \xi+\int_0^t \int \mu(X(s),x)\P\bigl(X(s)\in dx\bigr)\,ds+W^0(t)\right\rVert^2\right]\right)^{\!\nicefrac{1}{2}} \\
&\leq  \lVert \xi\rVert+
 \int_0^t \left(\E\!\left[\left \lVert \int  \mu(X(s),x)
\P\bigl(X(s)\in dx\bigr)
\right\rVert^2\right]\right)^{\!\nicefrac{1}{2}}
\,ds +\left(\E\!\left[\left\lVert W^0(t)\right\rVert^2\right]\right)^{\!\nicefrac{1}{2}} \\
&
\leq  \lVert \xi\rVert+
\int_0^t \lVert \mu(0,0)\rVert+L\left(\E\!\left[\left\lVert X(s)\right\rVert^2\right]\right)^{\!\nicefrac{1}{2}} ds+\sqrt{td}\\
&= 
\int_0^t 
L
\left(\E\!\left[\left\lVert X(s)\right\rVert^2\right]\right)^{\!\nicefrac{1}{2}}ds
+\lVert \xi\rVert+ \lVert \mu(0,0)\rVert t+\sqrt{td}.\end{split}\label{e03}
\end{align} 
This, the fact that
$\forall\,m\in\N\colon X_{0,m}^0=0$,
 Gronwall's lemma,  and the fact that 
$\int_{0}^{T} \left(\E\!\left[\lVert X(t)\rVert^2\right]\right)^{\!\nicefrac{1}{2}}\,dt<\infty$
prove  for all $t\in[0,T]$  that
\begin{align}
&
\left(
\E \!\left[\left
\lVert 
X^0_{0,m}(t)-X(t)\right\rVert^2\right]\right)^{\!\nicefrac{1}{2}}=
\left(\E\!\left[\lVert X(t)\rVert^2\right]\right)^{\!\nicefrac{1}{2}}
\leq  \left[
\lVert \xi\rVert+ \lVert \mu(0,0)\rVert t+
\sqrt{td}\right]e^{Lt}.\label{e01}
\end{align}
This proves \eqref{t01c}.

\refstepcounter{step}\label{st01} 
 \emph{Step~\thestep}.\
We establish  measurability and distributional properties. 
First, the assumptions on measurablity,  \eqref{a05}, induction, and the fact that
 $\forall\,m\in\N,\theta\in\Theta\colon 
X_{0,m}^\theta=0
$ prove for all $m\in \N$, $n\in\N_0$, $\theta\in\Theta$ that
${X}_{n,m}^\theta$ is measurable.
Next, the fact that
$\forall\,m\in\N,\theta\in\Theta\colon X_{0,m}^\theta=0$,
\eqref{a05}, and induction prove for all $n\in\N_0$,  $m\in\N$,
$\theta\in\Theta$ that
\begin{align}\begin{split}
&\sigma\!\left(\left\{W^\theta(t),X^\theta_{\ell,m}(t)
\colon \ell\in \{0,1,\ldots,n\},t\in [0,T]\right\}\right)
\\
&\subseteq  \sigma \!\left(\left\{W^\theta(t), W^{(\theta,i,\nu)}(t),\unif^{(\theta,i,\nu)}\colon i\in \{0,1,\ldots,n\},\nu\in\Theta, t\in[0,T] \right\} \right).
\end{split}\label{a01}\end{align}
This and the fact that
$\forall\,m\in\N,\theta\in\Theta\colon 
X_{0,m}^\theta=0
$
 prove for all 
$n,m\in\N$,
$\theta\in\Theta$,
$k,\ell\in\N$,
$j\in\{\ell-1,\ell\}$ that
\begin{align}\begin{split}
&\sigma\!\left(\left\{X^{(\theta,n,k,\ell)}_{j,m}(t)
\colon t\in [0,T]\right\}\right)\\
&
\subseteq  \sigma \!\left(\left\{W^{(\theta,n,k,\ell)}(t), W^{(\theta,n,k,\ell,\nu)}(t),\unif^{(\theta,n,k,\ell,\nu)}\colon\nu\in\Theta, t\in[0,T] \right\} \right).
\end{split}\end{align}
This, \eqref{a01}, and the independence assumptions show
for all 
$n,m\in\N$,
$\theta\in\Theta$
that
\begin{equation}
(W^\theta,(X^\theta_{j,m})_{j\in \{0,1,\ldots,n-1\}}),\quad
(X^{(\theta,n,k,\ell)}_{\ell,m}, X^{(\theta,n,k,\ell)}_{\ell-1,m}),\quad 
\unif^{(\theta,n,k,\ell)},\quad 
k,\ell\in\N,
\end{equation}
are independent.
This, the fact that
$\forall\,m\in\N,\theta\in\Theta\colon 
X_{0,m}^\theta=0
$, \eqref{a05},
the disintegration theorem 
(see, e.g., \cite[Lemma 2.2]{HJK+18}), and induction
show 
for all 
$m\in\N$,
$n\in\N_0$  that
$
( W^{\theta},(X_{\ell,m}^{\theta})_{ \ell\in\{0,1,\ldots,n\}}
)$, $\theta\in\Theta$, are identically distributed.

\refstepcounter{step}\label{st02} 
\emph{Step~\thestep.}\ 
We establish  that the approximations are square-integrable. Observe that the triangle inequality,
\eqref{a03},  distributional properties (see Step~\ref{st01}), and
the disintegration theorem 
(see, e.g., \cite[Lemma 2.2]{HJK+18}) prove for all
$n,m,k\in\N$, $\ell\in \{1,2,\ldots,n-1\}$, 
$j\in \{\ell-1,\ell\}$,
$t\in [0,T]$ that
\begin{align}
&
t\left(
\E
\!\left[
\left\lVert 
\mu\bigl(X^0_{j,m}(t \unif^{(0,n,k,\ell)}),
X^{(0,n,k,\ell)}_{j,m}(t \unif^{(0,n,k,\ell)})
\bigr)\right\rVert^2\right]\right)^{\!\nicefrac{1}{2}}\nonumber \\
&\leq  t\left\lVert \mu(0,0)\right\rVert
+\tfrac{tL}{2}
\left(\E\! \left[\bigl\lVert 
X^0_{j,m}(t \unif^{(0,n,k,\ell)})\bigr\rVert^2\right]\right)^{\!\nicefrac{1}{2}}+\tfrac{tL}{2} 
\left(
\E\! \left[\bigl\lVert 
X^{(0,n,k,\ell)}_{j,m}(t \unif^{(0,n,k,\ell)})
\bigr\rVert^2\right]\right)^{\!\nicefrac{1}{2}}\nonumber \\
&=  t\left\lVert \mu(0,0)\right\rVert
+  L\sqrt{t}
\left(t\E\! \left[\bigl\lVert 
X^0_{j,m}(t \unif^{(0,n,k,\ell)})\bigr\rVert^2\right]\right)^{\!\nicefrac{1}{2}}
\nonumber \\
&=t
\left\lVert \mu(0,0)\right\rVert
+  L\sqrt{t}
\left(\int_{0}^{t}\E\! \left[\bigl\lVert 
X^0_{j,m}(s)\bigr\rVert^2\right]\!ds\right)^{\!\nicefrac{1}{2}}.\label{a16}
\end{align}
This, the fact that
$\forall\,m\in\N,\theta\in\Theta\colon 
X_{0,m}^\theta=0
$, \eqref{a05}, the triangle inequality, and induction yield
for all
$n,m,k\in\N$, $\ell\in \{1,2,\ldots,n-1\}$,
$j\in \{\ell-1,\ell\}$
that
\begin{align}\small\begin{split}
&
\sup_{t\in[0,T]}\left[
\left(\E\!\left[
\left\lVert 
X_{n,m}^0(t)\right\rVert^2\right]\right)^{\!\nicefrac{1}{2}}+
t\left(
\E
\!\left[
\left\lVert 
\mu\bigl(X^0_{j,m}(t \unif^{(0,n,k,\ell)}),
X^{(0,n,k,\ell)}_{j,m}(t \unif^{(0,n,k,\ell)})
\bigr)\right\rVert^2\right]\right)^{\!\nicefrac{1}{2}}\right]<\infty.\end{split}
\label{a04}
\end{align}

\refstepcounter{step}\label{st03} 
\emph{Step~\thestep}.\ We consider the bias.
Observe that 
\eqref{a05}, 
\eqref{a04},
linearity of conditional expectations, the
 disintegration theorem (see, e.g., \cite[Lemma 2.2]{HJK+18}),   distributional properties
(cf. Step~\ref{st01}), a telescoping sum argument, 
the fact that
$\forall\,m\in\N\colon X_{0,m}^0=0$,
and the substitution rule
imply that for all
$n,m\in\N$, 
$t\in [0,T]$ it holds a.s.\ that
\begin{align}
&\E \Bigl[
X_{n,m}^0(t)-\xi-W^0\left(\sup\!\left(\{\tfrac{kT}{m^n}\colon k\in\N_0\}\cap[0,t]\right)\right)\Big|\mathcal{G}_{n-1,m} \Bigr]
\nonumber \\ 
&
= t\mu(0,0)+
\sum_{\ell=1}^{n-1}\Biggl(\frac{t}{m^{n-\ell}}
\sum_{k=1}^{m^{n-\ell}}
\biggl(
\E\!\left[
\mu\bigl(X^0_{\ell,m}(\unif^{(0,n,k,\ell)}t),
X^{(0,n,k,\ell)}_{\ell,m}(\unif^{(0,n,k,\ell)}t)
\bigr)\middle|\mathcal{G}_{n-1,m}\right]\nonumber \\
&\qquad\qquad\qquad\qquad\qquad\qquad\qquad
- \E\!\left[
\mu\bigl(X^0_{\ell-1,m}(\unif^{(0,n,k,\ell)}t),
X^{(0,n,k,\ell)}_{\ell-1,m}(\unif^{(0,n,k,\ell)}t)
\bigr)
\middle|\mathcal{G}_{n-1,m}
\right]
\biggr)
\Biggr)\nonumber \\
&
= t\mu(0,0)+
\sum_{\ell=1}^{n-1}\Biggl(\frac{1}{m^{n-\ell}}
\sum_{k=1}^{m^{n-\ell}}
\int_{0}^{t}
\E\!\left[
\mu\bigl(X^0_{\ell,m}(s),
X^{(0,n,k,\ell)}_{\ell,m}(s)
\bigr)\middle|\mathcal{G}_{n-1,m}\right]\nonumber \\
&\qquad\qquad\qquad\qquad\qquad\qquad\qquad
- \E\!\left[
\mu\bigl(X^0_{\ell-1,m}(s),
X^{(0,n,k,\ell)}_{\ell-1,m}(s)
\bigr)
\middle|\mathcal{G}_{n-1,m}
\right]\!ds
\Biggr)\nonumber \\
&=t\mu(0,0)+\sum_{\ell=1}^{n-1}
\int_{0}^{t}
\int
\mu\bigl(
X^{0}_{\ell,m}(s),y
\bigr)
\P\!\left( X^0_{\ell,m}(s)\in dy\right)\nonumber \\
&\qquad\qquad\qquad\qquad\qquad
- \int
\mu\bigl(
X^{0}_{\ell-1,m}(s),y
\bigr)
\P\!\left(X^0_{\ell-1,m}(s)\in dy\right)
ds \nonumber \\
&
=
\int_{0}^{t}\int
\mu\bigl(
X^{0}_{n-1,m}(s),y
\bigr)\P \!\left(
X^0_{n-1,m}(s)\in dy\right)ds
.\label{a18b}
\end{align}
Next, 
\eqref{a02} 
shows that for all $t\in [0,T]$ it holds a.s.\ that
\begin{align}\begin{split}
&
\E\! \left[X(t)-\xi-W^0(t)
\middle| \mathcal{G}_{n-1,m}
\right]
=X(t)-\xi-W^0(t)=\int_0^t \int \mu(X(s),x)\P\bigl(X(s)\in dx\bigr)\,ds.
\end{split}\end{align}
This, 
the triangle inequality,
the fact that
$\forall\,s,t\in [0,T]\colon \E\bigl[\lVert W^0(t)-W^0(s)\rVert^2\bigr]= d\lvert t-s\rvert$,
\eqref{a18b}, Jensen's inequality,  \eqref{a03},  and Tonelli's theorem show for all 
$n,m\in\N$,
$t\in [0,T]$ that
{\small
\begin{align}\label{a18}
&
\left(
\E\!
\left[
\left\lVert \E \!\left[
X_{n,m}^0(t)-
X(t)
\middle|\mathcal{G}_{n-1,m}\right]\right\rVert^2\right]\right)^{\!\nicefrac{1}{2}}
\leq \left(
\E\!
\left[
\left\lVert 
W^0\left(\sup\!\left(\{\tfrac{kT}{m^n}\colon k\in\N_0\}\cap[0,t]\right)\right)
- W^0(t)
\right\rVert^2\right]\right)^{\!\nicefrac{1}{2}}
\nonumber \\
&+ 
\Biggl(\E\biggl[\Bigl\lVert 
\E \Bigl[\Bigl(
X_{n,m}^0(t)-\xi-W^0\left(\sup\!\left(\{\tfrac{kT}{m^n}\colon k\in\N_0\}\cap[0,t]\right)\right)\Bigr)-
\Bigl(X(t)-\xi-W^0(t)\Bigr)
\Big|\mathcal{G}_{n-1,m} \Bigr]\Bigr\rVert^2\biggr]\Biggr)^{\!\nicefrac{1}{2}}\nonumber \\
&\leq \tfrac{\sqrt{Td}}{\sqrt{m^n}}+ \left(\E\!\left[\left\lVert \int_{0}^{t}\int\left[
\mu\bigl(X^0_{n-1,m}(s),
y
\bigr)
- \mu(X(s),x)\right]
\P\bigl(X^{0}_{n-1,m}(s)\in dy, X(s)\in dx\bigr)
\,ds\right\rVert^2\right]\right)^{\!\nicefrac{1}{2}}\nonumber \\
&\leq \tfrac{\sqrt{Td}}{\sqrt{m^n}}+\left(\E\!\left[t
\int_{0}^{t}
\int
\left\lVert
\mu\bigl(
X^{0}_{n-1,m}(s),y
\bigr)-\mu(X(s),x)\right\rVert^2
\P\bigl(X^{0}_{n-1,m}(s)\in dy, X(s)\in dx\bigr)\,ds
\right]\right)^{\!\nicefrac{1}{2}}\nonumber  \\
&\leq \tfrac{\sqrt{Td}}{\sqrt{m^n}}+
\tfrac{L\sqrt{t}}{2}
 \left(\E\!\left[
\int_{0}^{t}
\left\lVert 
X^{0}_{n-1,m}(s)-X(s)\right\rVert^2ds
\right]\right)^{\!\nicefrac{1}{2}}
+
\tfrac{L\sqrt{t}}{2}
 \left(\E\!\left[
\int_{0}^{t}
\E\!\left[\left\lVert 
X^{0}_{n-1,m}(s)-X(s)\right\rVert^2\right]\!ds
\right]\right)^{\!\nicefrac{1}{2}}\nonumber \\
&=\tfrac{\sqrt{Td}}{\sqrt{m^n}}+L\sqrt{t}
 \left(
\int_{0}^{t}\E\!\left[
\left\lVert 
X^{0}_{n-1,m}(s)-X(s)\right\rVert^2\right]\!ds
\right)^{\!\nicefrac{1}{2}}
.
\end{align}}
\refstepcounter{step}\label{st04} 
\emph{Step~\thestep}.\ We consider the statistical error. 
Distributional properties (cf. Step~\ref{st01}) imply  for all 
$n,m\in\N$, $\ell\in \{1,2,\ldots,n-1\}$ that
\begin{enumerate}[a)]\itemsep0pt
\item  it holds for all $k\in \N$ that
$(X^{(0,n,k,\ell)}_{\ell,m}, X^{(0,n,k,\ell)}_{\ell-1,m}, \unif^{(0,n,k,\ell)}) $ and 
$\mathcal{G}_{n-1,m}$ are independent,
\item 
it holds that
$(X^{(0,n,k,\ell)}_{\ell,m}, X^{(0,n,k,\ell)}_{\ell-1,m}, \unif^{(0,n,k,\ell)}) $, $k\in\N$,
are i.i.d., and
\item 
it holds  that
$(X^{(0,n,1,\ell)}_{\ell,m}, X^{(0,n,1,\ell)}_{\ell-1,m}) $
and 
$(X^{0}_{\ell,m}, X^{0}_{\ell-1,m}) $ are identically distributed.
\end{enumerate}
This,
\eqref{a05},
the triangle inequality,
Biennaym\'e's identity,
the assumptions on distributions,  and 
the disintegration theorem (see, e.g., \cite[Lemma 2.2]{HJK+18})
prove that for all 
$n,m\in\N$, 
$t\in [0,T]$ it holds a.s.\
that
\begin{align}
&
\left(\var\! \left[X_{n,m}^0(t)\middle|\mathcal{G}_{n-1,m}\right]\right)^{\!\nicefrac{1}{2}}=\Biggl(\var \Biggl[ \sum_{\ell=1}^{n-1}\biggl[\tfrac{t}{m^{n-\ell}}
\sum_{k=1}^{m^{n-\ell}}\Bigl[
\mu\bigl(X^0_{\ell,m}(\unif^{(0,n,k,\ell)}t),
X^{(0,n,k,\ell)}_{\ell,m}(\unif^{(0,n,k,\ell)}t)
\bigr)\nonumber \\[-5pt]
&\qquad\qquad\qquad\qquad\qquad\qquad
- 
\mu\bigl(X^0_{\ell-1,m}(\unif^{(0,n,k,\ell)}t),
X^{(0,n,k,\ell)}_{\ell-1,m}(\unif^{(0,n,k,\ell)}t)
\bigr)\Bigr]\biggr]\Big|\mathcal{G}_{n-1,m}\Biggr]\Biggr)^{\!\nicefrac{1}{2}}\nonumber \\
&\leq \sum_{\ell=1}^{n-1}\Biggl[\tfrac{t}{m^{n-\ell}}\Biggl(\var \Biggl[ 
\sum_{k=1}^{m^{n-\ell}}\Bigl[
\mu\bigl(X^0_{\ell,m}(\unif^{(0,n,k,\ell)}t),
X^{(0,n,k,\ell)}_{\ell,m}(\unif^{(0,n,k,\ell)}t)
\bigr)\nonumber \\[-10pt]
&\qquad\qquad\qquad\qquad\qquad\qquad
- 
\mu\bigl(X^0_{\ell-1,m}(\unif^{(0,n,k,\ell)}t),
X^{(0,n,k,\ell)}_{\ell-1,m}(\unif^{(0,n,k,\ell)}t)
\bigr)\Bigr]\bigg|\mathcal{G}_{n-1,m}\Biggr]\Biggr)^{\!\nicefrac{1}{2}}\Biggr]\nonumber \\
&\leq \sum_{\ell=1}^{n-1}\Biggl[\tfrac{\sqrt{t}}{\sqrt{m^{n-\ell}}}\biggl(t\E \biggl[ \Bigl\lVert 
\mu\bigl(X^0_{\ell,m}(\unif^{(0,n,1,\ell)}t),
X^{(0,n,1,\ell)}_{\ell,m}(\unif^{(0,n,1,\ell)}t)
\bigr)\nonumber \\[-10pt]
&\qquad\qquad\qquad\qquad\qquad
- 
\mu\bigl(X^0_{\ell-1,m}(\unif^{(0,n,1,\ell)}t),
X^{(0,n,1,\ell)}_{\ell-1,m}(\unif^{(0,n,1,\ell)}t)
\bigr)\Bigr\rVert^2\bigg|\mathcal{G}_{n-1,m}\biggr]\biggr)^{\!\nicefrac{1}{2}}\Biggr]
\end{align}
and
\begin{align}
&\left(\var\! \left[X_{n,m}^0(t)\middle|\mathcal{G}_{n-1,m}\right]\right)^{\!\nicefrac{1}{2}}\nonumber \\
&\leq\sum_{\ell=1}^{n-1}\Biggl[
\tfrac{\sqrt{t}}{\sqrt{m^{n-\ell}}}\biggl(
\int_{0}^{t}\int 
\left\lVert 
\mu\bigl(X_{\ell,m}^{0}(s),x\bigr)- \mu\bigl(X^{0}_{\ell-1,m}(s), y\bigr)\right\rVert^2\nonumber \\
&\qquad\qquad\qquad\qquad\qquad\qquad\qquad\qquad 
\P\!\left( X_{\ell,m}^{(0,n,1,\ell)}(s)\in dx, X_{\ell-1,m}^{(0,n,1,\ell)}(s) \in dy\right)
 ds\biggr)^{\!\nicefrac{1}{2}}\Biggr]\nonumber \\
&=\sum_{\ell=1}^{n-1}\Biggl[
\tfrac{\sqrt{t}}{\sqrt{m^{n-\ell}}}\biggl(
\int_{0}^{t}\int 
\left\lVert 
\mu\bigl(X_{\ell,m}^{0}(s),x\bigr)- \mu\bigl(X^{0}_{\ell-1,m}(s), y\bigr)\right\rVert^2\nonumber \\
&\qquad\qquad\qquad\qquad\qquad\qquad\qquad\qquad 
\P\!\left( X_{\ell,m}^{0}(s)\in dx, X_{\ell-1,m}^{0}(s) \in dy\right)
 ds\biggr)^{\!\nicefrac{1}{2}}\Biggr]\nonumber \\
&\leq \sum_{\ell=1}^{n-1}\Biggl[
\tfrac{L \sqrt{t}}{2\sqrt{m^{n-\ell}}}\left(
\int_{0}^{t}
\left\lVert X_{\ell,m}^{0}(s)-
X^{0}_{\ell-1,m}(s) \right\rVert^2ds
\right)^{\!\nicefrac{1}{2}}
\nonumber \\
&\qquad\qquad\qquad\qquad\qquad+\tfrac{L \sqrt{t}}{2\sqrt{m^{n-\ell}}}\left(
\int_{0}^{t}\E\!\left[
\left\lVert X_{\ell,m}^{0}(s)-
X^{0}_{\ell-1,m}(s) \right\rVert^2\right]ds\right)^{\!\nicefrac{1}{2}}
\Biggr].\label{a14}
\end{align}
This, the tower property, the definition of conditional variances, the triangle inequality,
  Jensen's inequality, and Tonelli's theorem
show 
for all $n,m\in \N $, $t\in[0,T]$ that
\begin{align}
&
\left(
\E\!
\left[\left\lVert X_{n,m}^0(t)-
\E \!\left[
X_{n,m}^0(t)
\middle|\mathcal{G}_{n-1,m}\right]\right\rVert^2\right]\right)^{\!\nicefrac{1}{2}}\nonumber \\
&=
\left(\E\!\left[
\E\!
\left[\left\lVert X_{n,m}^0(t)-
\E \!\left[
X_{n,m}^0(t)
\middle|\mathcal{G}_{n-1,m}\right]\right\rVert^2\middle|\mathcal{G}_{n-1,m} \right]\right]\right)^{\!\nicefrac{1}{2}}
= \left(\E\!\left[\var\! \left[X_{n,m}^0(t)\middle|\mathcal{G}_{n-1,m}\right]\right]\right)^{\!\nicefrac{1}{2}}
\nonumber \\
&
\leq  \sum_{\ell=1}^{n-1}\Biggl[
\frac{L \sqrt{t}}{\sqrt{m^{n-\ell}}}\left(
\int_{0}^{t}\E\!\left[
\left\lVert X_{\ell,m}^{0}(s)-
X^{0}_{\ell-1,m}(s) \right\rVert^2\right]\!ds
\right)^{\!\nicefrac{1}{2}}\Biggl]\nonumber \\
&\leq \sum_{\ell=0}^{n-1}
\frac{(2-\1_{\{n-1\}}(\ell))L\sqrt{t}}{\sqrt{m^{n-\ell-1}}}
\left(\int_{0}^{t}
\E \!\left[
\left
\lVert 
X^{0}_{\ell,m}(s)
-X(s)\right\rVert^2\right]\!ds\right)^{\!\nicefrac{1}{2}}.\label{a22}
\end{align}

\refstepcounter{step}\label{st06} 
\emph{Step~\thestep.}\ We now prove \eqref{t01a}.
Observe that the definition of $\mathcal{G}_{n,m}$, $n\in\N_0$, $m\in\N$, 
and the fact that $X$ is $(\sigma(\{W^0(s)\colon s\in [0,t]\}))_{t\in[0,T]}$-adapted  show
 for all $n,m\in\N$, $t\in[0,T]$ that
\begin{align}
&X_{n,m}^0(t)-X(t)= \E \!\left[
X_{n,m}^0(t)-
X(t)
\middle|\mathcal{G}_{n-1,m}\right]
+
X_{n,m}^0(t)-
\E \!\left[
X_{n,m}^0(t)
\middle|\mathcal{G}_{n-1,m}\right].
\label{a23}\end{align}
This,  the triangle inequality, \eqref{a18}, and \eqref{a22} show
for all $n,m\in\N$, $t\in[0,T]$ that
\begin{align}
&
\left(
\E\!\left[
\left \lVert 
X_{n,m}^0(t)-X(t)\right\rVert^2\right]\right)^{\!\nicefrac{1}{2}}\nonumber \\
&
\leq 
\left(
\E\!\left[
\left\lVert 
\E \!\left[
X_{n,m}^0(t)-
X(t)
\middle|\mathcal{G}_{n-1,m}\right]\right\rVert^2\right]\right)^{\!\nicefrac{1}{2}}
+ \left(\E\!\left[\left\lVert 
X_{n,m}^0(t)-
\E \!\left[
X_{n,m}^0(t)
\middle|\mathcal{G}_{n-1,m}\right]\right\rVert^2\right]\right)^{\!\nicefrac{1}{2}}
\nonumber \\
&\leq 
\tfrac{\sqrt{Td}}{\sqrt{m^n}}+ L\sqrt{t}
 \left(
\int_{0}^{t}
\E\!\left[\left\lVert 
X^{0}_{n-1,m}(s)-X(s)\right\rVert^2\right]\!ds
\right)^{\!\nicefrac{1}{2}}\nonumber \\
&\qquad
+\sum_{\ell=0}^{n-1}\left[
\tfrac{(2-\1_{\{n-1\}}(\ell))L\sqrt{t}}{\sqrt{m^{n-\ell-1}}}
\left(\int_{0}^{t}
\E \!\left[
\left
\lVert 
X^{0}_{\ell,m}(s)
-X(s)\right\rVert^2\right]\!ds\right)^{\!\nicefrac{1}{2}}\right]\nonumber \\
&=  \tfrac{\sqrt{Td}}{\sqrt{m^{n}}}
+\sum_{\ell=0}^{n-1}\left[
\tfrac{2L\sqrt{t}}{\sqrt{m^{n-\ell-1}}}
\left(\int_{0}^{t}
\E \!\left[
\left
\lVert 
X^{0}_{\ell,m}(s)
-X(s)\right\rVert^2\right]\!ds\right)^{\!\nicefrac{1}{2}}\right].
\end{align}
This, 
\cite[Lemma~3.9]{HJKN20}
(applied for every
$m,N\in\N $, $t\in [0,T]$
 with 
$a\gets \sqrt{Td}$, 
$ b\gets 2 L\sqrt{t}$,
$ c\gets 1/\sqrt{m}$,
$  \alpha\gets 0$,
$\beta\gets t$,
$p\gets 2$,
$
(f_n)_{n\in\N_0}\gets 
\bigl([0,t]\ni s\mapsto
\bigl(
\E \bigl[
\lVert 
X^0_{n,m}(s)-X(s)\rVert^2\bigr]\bigr)^{\!\nicefrac{1}{2}}\in [0,\infty]\bigr)_{n\in\N_0}$ in the notation of \cite[Lemma~3.9]{HJKN20}), \eqref{e01}, and
the fact that for all $m,N\in\N$ it holds that
\begin{equation}
 \max_{k\in\{0,1,\ldots,N\}} \tfrac{1}{\sqrt{m^{N-k}k!}}= m^{-N/2} \max_{k\in\{0,1,\ldots,N\}} \tfrac{\sqrt{m^k}}{\sqrt{k!}}\leq m^{-N/2}e^{m/2}
\end{equation}
show for all 
$t\in[0,T]$, $m,N\in\N$ that
\begin{align}\begin{split}
&
\left(
\E \!\left[\left
\lVert 
X^0_{N,m}(t)-X(t)\right\rVert^2\right]\right)^{\!\nicefrac{1}{2}}\leq \left[
\sqrt{Td}
+ 2L\sqrt{t}\cdot \sqrt{t}
\sup_{s\in [0,t]}\left(
\E \!\left[\left
\lVert 
X^0_{0,m}(s)-X(s)\right\rVert^2\right]\right)^{\!\nicefrac{1}{2}}
\right]\\
&\qquad\qquad\cdot\left[\max_{k\in\{0,1,\ldots,N\}} \tfrac{1}{\sqrt{m^{N-k}k!}}\right]
 \left(1+2L\sqrt{t}\cdot \sqrt{t}\right)^{N-1}\\
&\leq \left[
\sqrt{Td}
+ 2Lt\left[
\lVert \xi\rVert+ \lVert \mu(0,0)\rVert t+
\sqrt{td}\right]e^{Lt}
\right] m^{-N/2}e^{m/2}\left(1+2Lt\right)^{N-1}\\
&\leq \left[
\sqrt{Td}
+2Lt\left[
\lVert \xi\rVert+ \lVert \mu(0,0)\rVert t+
\sqrt{td}\right]
\right]e^{Lt} m^{-N/2}e^{m/2}\left(1+2Lt\right)^{N-1}\\
&\leq m^{-N/2}e^{m/2}
\left[
\lVert \xi\rVert+ \lVert \mu(0,0)\rVert t+
\sqrt{Td}\right]e^{Lt}
\left(1+2Lt\right)^{N}.
\end{split}\end{align}
This proves \eqref{t01a}.

\refstepcounter{step}\label{st07} 
 \emph{Step~\thestep.}\ We estimate the computational complexity. 
Let $\mathrm{n}=(\mathrm{n}_\epsilon)_{\epsilon\in (0,1)}\colon (0,1) \to \N\cup \{\infty\}$ satisfy 
for all $\epsilon\in(0,1)$
that
\begin{equation}
\mathrm{n}_{\epsilon} = \inf\left(\left\{n\in\N\colon \sup_{k\in[n,\infty)\cap\N}\sup_{t\in[0,T]}
\left(
\E \!\left[\left
\lVert 
X^0_{k,k}(t)-X(t)\right\rVert^2\right]\right)^{\!\nicefrac{1}{2}}<\epsilon
\right\}\cup\{\infty\}\right).
\end{equation}
This, \eqref{t01a},
 and the fact that
$\forall\,\delta \in (0,1),s,t\in (0,\infty)\colon \lim_{n\to\infty}(s^n n^t n^{-n/2})=0$
prove for all $\epsilon\in (0,1)$ that
$\lim_{n\to\infty}\sup_{t\in[0,T]}\E \!\left[
\lVert 
X^0_{n,n}(t)-X(t)\rVert^2\right]=0$ and 
 $\mathrm{n}_{\epsilon}\in\N$.
Next,  \eqref{a06} and the fact that
$\forall\,\ell,m\in\N\colon m^{-\ell}\leq 1$
imply for all $m,n\in\N$ that $C_{0,m}=0$ and
\begin{equation}\begin{split}
&m^{-n}C_{n,m} \leq 
\mathfrak{v}d+\mathfrak{f}m^{-n}+
\sum_{\ell=1}^{n-1}\Bigl[ m^ {-\ell}\Bigl( \mathfrak{v}(m^\ell d+1) +2 \mathfrak{f}
+2 C_{\ell,m}+2  C_{\ell-1,m}\Bigr)\Bigr]\\
&\leq 2n(\mathfrak{v}d+\mathfrak{f})
+\sum_{\ell=0}^{n-1} \Bigl(2m^{-\ell} C_{\ell,m}+\1_{\N}(\ell)2m^{-(\ell-1)}C_{\lvert \ell-1\rvert,m}\Bigr)
.
\end{split}\end{equation}
This,
\cref{l05} (applied for every $m\in \N$
 with
$(a_n)_{n\in\N_0}\gets (m^{-n}C_{n,m})_{n\in \N_0}$,
$\kappa\gets 2$,
 $\lambda\gets 2$,
$c_1\gets 0$,
$c_2\gets 2(\mathfrak{v}d+\mathfrak{f})$,
$c_3\gets 0$, 
$c_4\gets 1$, $\beta\gets \tfrac{(1+2)+\sqrt{(1+2)^2+4\cdot 2}}{2}$
in the notation of \cref{l05}), and the fact that
$
\tfrac{(1+2)+\sqrt{(1+2)^2+4\cdot 2}}{2}\leq 4
$
 imply for all $n,m\in\N$ that
$
m^{-n}C_{n,m}\leq 
1.5 \cdot 2(\mathfrak{v}d+\mathfrak{f})\frac{4^n-1}{4-1}
\leq (\mathfrak{v}d+\mathfrak{f})4^n
$.
This and \eqref{t01a} imply for all $n,m\in\N$ that
$C_{n,m}\leq (\mathfrak{v}d+\mathfrak{f})(4m)^n$
 and
\begin{equation}\small\begin{split}
&C_{n+1,n+1}\sup_{t\in[0,T]}
\left(
\E \!\left[\left
\lVert 
X^0_{n,n}(t)-X(t)\right\rVert^2\right]\right)^{\frac{2+\delta}{2}}\\
&\leq (\mathfrak{v}d+ \mathfrak{f}) (4n+4)^{n+1}\left[
n^{-n/2}e^{n/2}
\left[
\lVert \xi\rVert+ \lVert \mu(0,0)\rVert T+
\sqrt{dT}\right]e^{LT}
\left(1+2LT\right)^{n}\right]^{2+\delta}\\
&\leq (\mathfrak{v}d+ \mathfrak{f}) 
\sup_{k\in\N}\left[
(4 k+4)^{k+1}\left[
k^{-k/2}e^{k/2}
\left[
\lVert \xi\rVert+ \lVert \mu(0,0)\rVert T+
\sqrt{dT}\right]e^{LT}
\left(1+2LT\right)^{k}\right]^{2+\delta}\right].
\end{split}\end{equation}
This and the fact that
$\forall\,\delta \in (0,1),s,t\in (0,\infty)\colon \lim_{n\to\infty}(s^n (n+1)^t n^{-n\delta/2})=0$ imply
 for all $\delta,\epsilon\in (0,1)$   that
\begin{align}\begin{split}
&C_{\mathrm{n}_\epsilon,\mathrm{n}_\epsilon} \epsilon^{2+\delta}\leq
 \1_{\{1\}}(\mathrm{n}_\epsilon)
C_{1,1}+
 \1_{[2,\infty)}(\mathrm{n}_\epsilon)C_{\mathrm{n}_\epsilon,\mathrm{n}_\epsilon}\sup_{t\in[0,T]}
\left(
\E \!\left[\left
\lVert 
X^0_{\mathrm{n}_\epsilon-1,\mathrm{n}_\epsilon-1}(t)-X(t)\right\rVert^2\right]\right)^{\frac{2+\delta}{2}}\\
&\leq  
(\mathfrak{v}d+ \mathfrak{f})
\sup_{k\in\N}\left[
(4k+4)^{k+1}\left[\tfrac{
e^{k/2}
\left[1+
\lVert \xi\rVert+ \lVert \mu(0,0)\rVert T+
\sqrt{Td}\right]e^{LT}
\left(1+2LT\right)^{k}}{k^{k/2}}\right]^{2+\delta}\right]<\infty.\end{split}
\end{align}
This proves \eqref{t01b}. The proof of \cref{t01} is thus completed. 
\end{proof}

\subsubsection*{Acknowledgements}
We thank Arnulf Jentzen for very helpful comments and suggestions. 
This work has been funded by the Deutsche Forschungsgemeinschaft (DFG, German Research Foundation) through the research grant HU1889/6-2.

{
\bibliographystyle{acm}
\bibliography{./bibfile}

}

\end{document}